\documentclass[11pt]{article}
\usepackage{t1enc}
\usepackage[latin1]{inputenc}
\usepackage[english]{babel}
\usepackage{amsmath,amsthm}
\usepackage{amsfonts}
\usepackage{latexsym}
\usepackage[dvips]{graphicx}
\usepackage{graphicx}
\usepackage{mathrsfs}
\title{ Bounds for the rank of a complex unit gain graph in terms of the independence number}
\author{ Shengjie  He$^1$, Rong-Xia Hao$^1$\footnote{Corresponding author.
Emails: he1046436120@126.com (Shengjie  He), rxhao@bjtu.edu.cn (Rong-Xia Hao), yuaimeimath@163.com (Aimei Yu)}, Aimei Yu$^1$\\
{\small\em 1. Department of Mathematics, Beijing Jiaotong University, Beijing,
100044, China}\\
  }

\date{} \textwidth 16cm \textheight 22cm \topmargin 0 cm \hoffset
-1.5 cm \voffset 0cm

\newtheorem{theorem}{Theorem}[section]
\newtheorem{lemma}[theorem]{Lemma}

\newtheorem{definition}[theorem]{Definition}

\begin{document}
\baselineskip 0.50cm \maketitle

\begin{abstract}
A complex unit gain graph (or $\mathbb{T}$-gain graph) is a triple $\Phi=(G, \mathbb{T}, \varphi)$
($(G, \varphi)$ for short) consisting of a graph $G$ as the underlying graph of $(G, \varphi)$, $\mathbb{T}= \{ z \in C:|z|=1 \} $ is a subgroup of the multiplicative group of all nonzero complex numbers $\mathbb{C}^{\times}$ and a gain function $\varphi: \overrightarrow{E} \rightarrow \mathbb{T}$ such
that $\varphi(e_{ij})=\varphi(e_{ji})^{-1}=\overline{\varphi(e_{ji})}$.
In this paper, we investigate the relation among the rank, the independence number
and the cyclomatic number of a complex unit gain graph $(G, \varphi)$ with order $n$, and prove that $2n-2c(G) \leq r(G, \varphi)+2\alpha(G) \leq 2n$. Where $r(G, \varphi)$, $\alpha(G)$ and $c(G)$ are the rank of the Hermitian adjacency matrix $A(G, \varphi)$, the independence number and the cyclomatic number of $G$, respectively.
Furthermore, the properties of the complex unit gain graph that reaching the lower bound are
characterized.


\end{abstract}

\section{Introduction}
The study of the spectral properties of a graph is a popular subject in the graph theory.
The relation among the rank of the adjacent matrix and other topological structure parameters of a graph
has been studied extensively by many researchers.
Recently there has been a growing study of the rank of the adjacent matrix associated to
signed graphs and mixed graphs. In this paper we characterize the properties of the rank of a complex unit
gain graph. We refer to \cite{BONDY} for undefined terminologies and notation.

In this paper, we only consider the simple and finite graphs.
Let $G$ be an undirected graph with vertex set $V(G)=\{ v_{1}, v_{2}, \cdots, v_{n} \}$. The $degree$ of a vertex $u \in V(G)$, denote by $d_{G}(u)$, is the number of vertices
which are adjacent to $u$. A vertex of $G$ is called a {\it pendant vertex} if it is a vertex of degree one in $G$, whereas a vertex of $G$ is called a {\it quasi-pendant vertex} if it is adjacent to
a pendant vertex in $G$ unless it is a pendant vertex.
Denote by $P_n$ and $C_n$ a path and cycle on $n$ vertices, respectively.
The {\it adjacency matrix} $A(G)$ of $G$ is the $n \times n$ matrix whose $(i, j)$-entry equals to
1 if vertices $v_{i}$ and $v_j$ are adjacent and 0 otherwise.

A $complex$ $unit$ $gain$ $graph$ (or $\mathbb{T}$-gain graph) is a graph with the additional structure that
each orientation of an edge is given a complex unit, called a $gain$, which is the inverse of the
complex unit assigned to the opposite orientation. For a simple graph $G$ with order $n$,
let $\overrightarrow{E}$ be the set of oriented edges, it is obvious that this set contains two copies of each edge with opposite directions. We write $e_{ij}$ for the oriented edge from $v_{i}$ to $v_{j}$. The circle group, which is denoted by $\mathbb{T}= \{ z \in C:|z|=1 \} $, is a subgroup of the multiplicative group of all
nonzero complex numbers $\mathbb{C}^{\times}$. A complex unit gain graph is
a triple $\Phi=(G, \mathbb{T}, \varphi)$ consisting of a graph $G$, $\mathbb{T}= \{ z \in C:|z|=1 \} $ is a subgroup of the multiplicative group of all nonzero complex numbers $\mathbb{C}^{\times}$ and a gain function $\varphi: \overrightarrow{E} \rightarrow \mathbb{T}$, where $G$ is the underlying graph of $\Phi$ and   $\varphi(e_{ij})=\varphi(e_{ji})^{-1}=\overline{\varphi(e_{ji})}$.
For convenience, we write $(G, \varphi)$ for a complex unit gain graph $\Phi=(G, \mathbb{T}, \varphi)$ in this paper. The adjacency matrix associated to the complex unit gain graph $(G, \varphi)$ is the $n \times n$ complex matrix $A(G, \varphi)=a_{ij}$, where $a_{ij}=\varphi(e_{ij})$ if $v_{i}$ is adjacent to $v_{j}$, otherwise $a_{ij}=0$. It is obvious to see that $A(G, \varphi)$ is Hermitian and its eigenvalues are real. If the gain of every edge is 1 in $(G, \varphi)$, then the adjacency matrix $A(G, \varphi)$ is exactly the adjacency matrix $A(G)$ of the underlying graph $G$. It is obvious that a simple graph is assumed as a
complex unit gain graph with all positive gain 1's.
The $positive$ $inertia$ $index$, denoted by $p^{+}(G, \varphi)$, and the $negative$ $inertia$ $index$, denoted by $n^{-}(G, \varphi)$, of a complex unit gain graph $(G, \varphi)$ are defined to be the number of positive eigenvalues and negative eigenvalues of $A(G, \varphi)$ including multiplicities, respectively.
The $rank$ of a complex unit gain graph $(G, \varphi)$,
written as $r(G, \varphi)$, is defined to be the rank of $A(G, \varphi)$.
Obviously, $r(G, \varphi)=p^{+}(G, \varphi)+n^{-}(G, \varphi)$.

For an induced subgraph $H$ of a graph $G$,
denote by $G-H$, the subgraph obtained from $G$ by deleting all vertices of $H$ and all incident edges.
For a subset $X$ of $V(G)$, $G-X$ is the induced subgraph obtained from $G$ by deleting all vertices in $X$ and
all incident edges. In particular, $G-\{ x \}$ is usually written as $G-x$ for simplicity.
For an induced subgraph $H$ and a vertex $u$
outside $H$, the induced subgraph of $G$ with vertex set $V(H) \cup \{ u \}$ is simply written as $H+u$.

For a graph $G$, let $c(G)$ be the {\it cyclomatic number} of $G$,
that is $c(G)=|E(G)|-|V(G)|+\omega(G)$, where $\omega(G)$ is the number of connected components of $G$.
Two vertices of a graph $G$ are said to be $independent$ if they are not adjacent. A subset $I$ of
$V(G)$ is called an $independent$ $set$ if any two vertices of $I$ are independent in $G$. An independent set $I$ is $maximum$ if $G$ has no independent set $I'$ with $|I'|> |I'|$. The number
of vertices in a maximum independent set of $G$ is called the $independence$ $number$
of $G$ and is denoted by $\alpha(G)$.
For a complex unit gain graph $(G, \varphi)$, the independence number and cyclomatic number
 of $(G, \varphi)$  are defined to be the independence number and cyclomatic number of its underlying graph, respectively.

Let $G$ be a graph with pairwise vertex-disjoints cycles (if any) and $\mathscr{C}_{G}$ be the set of all cycles of $G$.
$T_{G}$ is an acyclic graph obtained from $G$ by contracting each cycle of $G$ into a vertex (called a {\it cyclic vertex}).
Denoted by $\mathscr{O}_{G}$ the set of all cyclic vertex of $G$.
Moreover, denoted by $[T_{G}]$ the subgraph of $T_{G}$
induced by all non-cyclic vertices. It is obviously that $[T_{G}]=T_{G}-\mathscr{O}_{G}$.

The rank of graphs have been discussed intensively by many researchers.
There are some papers focused on the study on the rank of graphs in terms of other topological structure parameters.
Wang and Wong characterized the bounds for the matching number, the edge chromatic number and the
independence number of a graph in terms of rank in \cite{WANGLONG}.
Gutman and  Sciriha \cite{GUT} studied the nullity of line graphs of trees.
Guo et al. \cite{MOHAR} and Liu et al. \cite{LXL} introduced the Hermitian adjacency matrix of a mixed graph and presented some basic properties of the rank of the mixed graphs independently.
In \cite{FYZ1}, the rank of the signed unicyclic graph was discussed by Fan et al.
He et al. characterized the relation among the rank,
the matching number and the cyclomatic number of a signed graph in \cite{HSJ}.
Chen et al. \cite{LSC} investigated the relation between the $H$-rank of a mixed graph and the
matching number of its underlying graph.
For other research of the rank of a graph one may be referred to those in \cite{BEVI,HOU,MAH2,MOHAR2,WANGXIN}.

Recently, the study of the properties of complex unit gain graphs has attracted increased attention.
Reff extended some fundamental concepts from spectral graph theory to complex unit gain graphs and
defined the adjacency, incidence and Laplacian matrices of them in \cite{REFF}.
Yu et al. \cite{YGH} investigated some properties of inertia of complex unit gain graphs and discussed the inertia index of a complex unit gain cycle.
In \cite{FYZUNIT}, Wang et al. provided a combinatorial description of the determinant of the Laplacian matrix of a complex unit gain graph which generalized that for the determinant of the Laplacian matrix of a signed graph.
Lu et al. \cite{LUY} studied the complex unit gain unicyclic graphs with small positive or
negative index and characterized the complex unit gain bicyclic graphs with rank 2, 3 or 4.
In \cite{WLG}, the relation among the rank of a complex unit gain graph and the rank of its underlying graph and the cyclomatic number was investigated by Lu et al.

In this paper, the upper and lower bounds of the rank of a complex unit gain graph $(G, \varphi)$ with order $n$
in terms of the cyclomatic number and the independence number of its underlying graph are investigated.
Moreover, the properties of the extremal graphs which attended the lower bound are identified.
The following Theorems \ref{T30} and \ref{T50} are our main results.

\begin{theorem}\label{T30}
Let $(G, \varphi)$ be a complex unit gain graph with order $n$. Then
$$2n-2c(G)-2\alpha(G) \leq r(G, \varphi) \leq 2n-2\alpha(G).$$
\end{theorem}

\begin{theorem}\label{T50}
Let $(G, \varphi)$ be a complex unit gain graph with order $n$. Then $ r(G, \varphi)= 2n-2c(G)-2\alpha(G)$
if and only if all the following conditions hold for $(G, \varphi)$:

{\em(i)} the cycles (if any) of $(G, \varphi)$ are pairwise vertex-disjoint;

{\em(ii)} for each cycle (if any) $(C_{l}, \varphi)$ of $(G, \varphi)$, either $\varphi(C_{l}, \varphi)=(-1)^{\frac{l}{2}}$ and $l$ is even or $Re((-1)^{\frac{l-1}{2}}\varphi(C_{l}, \varphi))=0$ and $l$ is odd;

{\em(iii)} $\alpha(T_{G})=\alpha([T_{G}])+c(G)$.
\end{theorem}

The rest of this paper is organized as follows. Prior to showing our main results, in Section 2,
we list some known elementary lemmas and results which will be useful.
In Section 3, we give the proof of the Theorem \ref{T30}. In Section 4, the properties of the extremal signed graphs which attained the lower bound of Theorem \ref{T30} are identified, and the proof of the Theorem \ref{T50} is presented.

\section{Preliminaries}

In this section, some known results and useful lemmas which will be
used in the proofs of our main results are listed.

\begin{lemma} \label{L16}{\rm\cite{YGH}}
Let $(G, \varphi)$ be a complex unit gain graph.

{\em(i)} If $(H, \varphi)$ is an induced subgraph of $(G, \varphi)$, then $r(H, \varphi) \leq r(G, \varphi)$.

{\em(ii)} If $(G_{1}, \varphi), (G_{2}, \varphi), \cdots, (G_{t}, \varphi)$ are all the connected components of $(G, \varphi)$, then $r(G, \varphi)=\sum_{i=1}^{t}r(G_{i}, \varphi)$.

{\em(iii)} $r(G, \varphi)\geq 0$ with equality if and only if $(G, \varphi)$ is an empty graph.
\end{lemma}

\begin{definition} \label{L053}{\rm\cite{LUY}}
Let $(C_{n}, \varphi)$ ($n \geq 3$) be a complex unit gain cycle and
$$\varphi(C_{n}, \varphi)=\varphi(v_{1}v_{2} \cdots v_{n}v_{1} )=\varphi(v_{1}v_{2})\varphi(v_{2}v_{3})  \cdots \varphi(v_{n-1}v_{n})\varphi(v_{n}v_{1}).$$
Then $(C_{n}, \varphi)$ is said to be one of the following five Types:
$$\left\{
    \begin{array}{ll}
      \rm{Type~A}, & \hbox{if $\varphi(C_{n}, \varphi)=(-1)^{\frac{n}{2}}$ and $n$ is even;} \\
      \rm{Type~B}, & \hbox{if $\varphi(C_{n}, \varphi) \neq (-1)^{\frac{n}{2}}$ and $n$ is even;} \\
      \rm{Type~C}, & \hbox{if $Re((-1)^{\frac{n-1}{2}}\varphi(C_{n}, \varphi))>0$ and $n$ is odd;} \\
      \rm{Type~D}, & \hbox{if $Re((-1)^{\frac{n-1}{2}}\varphi(C_{n}, \varphi))<0$ and $n$ is odd;} \\
      \rm{Type~E}, & \hbox{if $Re((-1)^{\frac{n-1}{2}}\varphi(C_{n}, \varphi))=0$ and $n$ is odd.}
    \end{array}
  \right.
$$
Where $Re(\cdot)$ is the real part of a complex number.
\end{definition}

\begin{lemma} \label{L12}{\rm\cite{YGH}}
Let $(C_{n}, \varphi)$ be a complex unit gain cycle of order $n$. Then
$$(p^{+}(C_{n}, \varphi), n^{-}(C_{n}, \varphi))=\left\{
             \begin{array}{ll}
               (\frac{n-2}{2}, \frac{n-2}{2}), & \hbox{if $(C_{n}, \varphi)$ is of \rm{Type~A};} \\
               (\frac{n}{2}, \frac{n}{2}), & \hbox{if $(C_{n}, \varphi)$ is of \rm{Type~B};} \\
               (\frac{n+1}{2}, \frac{n-1}{2}), & \hbox{if $(C_{n}, \varphi)$ is of \rm{Type~C};} \\
               (\frac{n-1}{2}, \frac{n+1}{2}), & \hbox{if $(C_{n}, \varphi)$ is of \rm{Type~D};} \\
               (\frac{n-1}{2}, \frac{n-1}{2}), & \hbox{if $(C_{n}, \varphi)$ is of \rm{Type~E}.}
             \end{array}
           \right.
$$
\end{lemma}

\begin{lemma} \label{L15}{\rm\cite{YGH}}
Let $(T, \varphi)$ be an acyclic complex unit gain graph.
Then $r(T, \varphi)= r(T)$.
\end{lemma}

From Lemma \ref{L15}, we have the following Lemma \ref{LPN} directly.
\begin{lemma} \label{LPN}
Let $(P_{n}, \varphi)$ be a complex unit gain path with order $n$.
Then
$$ r(P_{n}, \varphi)=\left\{
                      \begin{array}{ll}
                        n-1, & \hbox{if $n$ is odd;} \\
                        n, & \hbox{if $n$ is even.}
                      \end{array}
                    \right.
$$
\end{lemma}

\begin{lemma} \label{L051}{\rm\cite{BONDY}}
Let $T$ be a acyclic graph with order $n$.
Then $r(T)=2m(T)$ and $\alpha(T)+m(T)=n$.
\end{lemma}

Obviously, by Lemmas \ref{L15} and \ref{L051}, the following Lemma \ref{L052} can be obtained.

\begin{lemma} \label{L052}
Let $(T, \varphi)$ be an acyclic complex unit gain graph with order $n$.
Then $r(T, \varphi)+2\alpha(T)=2n$.
\end{lemma}

\begin{lemma} \label{L13}{\rm\cite{YGH}}
Let $y$ be a pendant vertex of a complex unit gain graph $(G, \varphi)$ and $x$ is the neighbour of $y$.
Then $r(G, \varphi)=r((G, \varphi)-  \{ x, y \} )+2$.
\end{lemma}

\begin{lemma} \label{L14}{\rm\cite{YGH}}
Let $x$ be a vertex of a complex unit gain graph $(G, \varphi)$.
Then $r(G, \varphi)-2 \leq r((G, \varphi)-x) \leq r(G, \varphi)$.
\end{lemma}

\begin{lemma} \label{L053}{\rm\cite{HLSC}}
Let $y$ be a pendant vertex of a graph $G$ and $x$ is the neighbour of $y$.
Then $\alpha(G)=\alpha(G-x)=\alpha(G-\{ x, y \})+1$.
\end{lemma}

\begin{lemma} \label{L23}{\rm\cite{WDY}}
Let $G$ be a graph with $x \in V(G)$.

{\em(i)} $c(G)=c(G-x)$ if $x$ lies outside any cycle of $G$;

{\em(ii)} $c(G-x) \leq c(G)-1$ if $x$ lies on a cycle of $G$;

{\em(iii)} $c(G-x) \leq c(G)-2$ if $x$ is a common vertex of distinct cycles of $G$.
\end{lemma}

\begin{lemma} \label{L054}{\rm\cite{HLSC}}
Let $G$ be a graph. Then

{\em(i)} $\alpha(G)-1  \leq \alpha(G-x) \leq \alpha(G) $ for any vertex $x \in V(G)$;

{\em(ii)} $\alpha(G-e) \geq \alpha(G)$ for any edge $e \in E(G)$.
\end{lemma}

\begin{lemma} \label{L055}{\rm\cite{HLSC}}
Let $T$ be a tree with at least one edge and $T_{0}$ be the subtree obtained
from $T$ by deleting all pendant vertices of $T$.

{\em(i)} $\alpha(T)  \leq \alpha(T_{0}) +p(T) $, where $p(T)$ is the number of pendent vertices of $T$;

{\em(ii)} If $\alpha(T) = \alpha(T-D)+|D|$ for a subset $D$ of $V(T)$, then there is a pendant vertex $x$ such
that $x \notin D$.
\end{lemma}

\section{Proof of Theorem \ref{T30}}

In this section, the proof for Theorem \ref{T30} is presented.

\noindent
{\bf The proof of Theorem \ref{T30}.}

Firstly, we show that $ r(G, \varphi)\leq 2n -2\alpha(G)$.
Let $I$ be a maximum independent set of $G$, i.e., $|I|=\alpha(G)$. Then

$$A(G, \varphi)=\left(
                \begin{array}{cc}
                  \mathbf{0} & \boldsymbol{B} \\
                  \boldsymbol{B}^{\top} & \boldsymbol{A} \\
                \end{array}
              \right)
$$
where $\boldsymbol{B}$ is a submatrix of $A(G, \varphi)$ with row indexed by $I$ and column indexed by $V(G)-I$,
$\boldsymbol{B}^{\top}$ refers to the transpose of $\boldsymbol{B}$ and $\boldsymbol{A}$ is the adjacency matrix of the induced subgraph $G-I$. Then it can be checked that
$$r(G, \varphi)\leq r(\mathbf{0}, \boldsymbol{B})+r(\boldsymbol{B}^{\top}, \boldsymbol{A})\leq n-\alpha(G)+n-\alpha(G)=2n-2\alpha(G).$$
Thus,
$$r(G, \varphi)\leq 2n-2\alpha(G).$$

Next, we argue by induction on $c(G)$ to show that $2n-2c(G) \leq r(G, \varphi)+2\alpha(G) $.
If $c(G)=0$, then $(G, \varphi)$ is a complex unit gain tree, and so result follows from Lemma \ref{L052}.
Hence one can assume that $c(G) \geq 1$. Let
$u$ be a vertex on some cycle of $(G, \varphi)$ and $(G', \varphi)=(G, \varphi)-u$.
Let $(G_{1}, \varphi), (G_{2}, \varphi), \cdots, (G_{l}, \varphi)$ be all
connected components of $(G', \varphi)$.
By Lemma \ref{L23}, we have
\begin{equation} \label{E1}
\sum_{i=1}^{l}c(G_{i})=c(G') \leq c(G)-1.
\end{equation}
By the induction hypothesis, one has
\begin{equation} \label{E2}
2(n-1)-2c(G') \leq r(G', \varphi)+2\alpha(G').
\end{equation}
By Lemmas \ref{L054} and \ref{L14}, we have
\begin{equation} \label{E3}
\sum_{i=1}^{l}\alpha(G_{i})=\alpha(G') \leq \alpha(G)
\end{equation}
and
\begin{equation} \label{E4}
\sum\limits_{i=1}^{l}r(G_{i}, \varphi)=r(G', \varphi) \leq r(G, \varphi).
\end{equation}
Thus the desired inequality now follows by combining (\ref{E1}), (\ref{E2}),  (\ref{E3}) and (\ref{E4}),
\begin{eqnarray} \label{E0}
r(G, \varphi)+2\alpha(G) & \ge & r(G', \varphi)+2\alpha(G')
\\ \nonumber
& \ge & 2(n-1)-2c(G')
\\ \nonumber
& \geq & 2(n - 1) - 2(c(G) - 1) = 2n-2c(G),
\end{eqnarray}
as desired.

This completes the proof of Theorem \ref{T30}.
$\square$

\section{Proof of Theorem \ref{T50}.}

A complex unit gain graph $(G, \varphi)$ with order $n$ is called {\it lower-optimal} if $r(G, \varphi)=2n-2c(G)-2\alpha(G)$,
or equivalently, the complex unit gain graph which attain the lower bound in Theorem \ref{T30}.
In this section, we characterize the properties of the complex unit gain graphs which are lower-optimal,
and then we give the proof for Theorem \ref{T50}.

\begin{lemma} \label{L001}
Let $u$ be a cut vertex of a complex unit gain graph $(G, \varphi)$ and $(H, \varphi)$ be a component
of $(G, \varphi)-u$. If $r(H, \varphi)=r((H, \varphi)+u)$, then $r(G, \varphi)=r(H, \varphi)+r((G, \varphi)-(H, \varphi))$.
\end{lemma}
\begin{proof}
Let $|V(H, \varphi)|=k$ and
$$A(G, \varphi) = \left(
      \begin{array}{ccc}
        \boldsymbol{A} & \mathbf{\beta} & \mathbf{0} \\
        \mathbf{\overline{\beta}^{\top}} & 0 & \mathbf{\gamma} \\
        \boldsymbol{0} &  \mathbf{\overline{\gamma}^{\top}} & \boldsymbol{B} \\
      \end{array}
    \right),
$$
where $\boldsymbol{A}$ and $\boldsymbol{B}$ are the Hermitian adjacency matrices of $(H, \varphi)$ and $(G, \varphi)-(H, \varphi)-u$, respectively. $\mathbf{\overline{\beta}^{\top}}$ refers to the conjugate transpose of $\mathbf{\beta}$. Since $r(H, \varphi)=r((H, \varphi)+u)$, the linear equation
$\boldsymbol{A}X=\mathbf{\beta}$ has solutions. Let $\xi$ be a solution of $\boldsymbol{A}X=\mathbf{\beta}$, and put
$$Q = \left(
      \begin{array}{ccc}
        \boldsymbol{E_{k}} & -\xi & \mathbf{0} \\
        \mathbf{0} & 1 & \mathbf{0} \\
        \mathbf{0} &  \mathbf{0} & \boldsymbol{I_{n-k-1}} \\
      \end{array}
    \right),
$$
where $\boldsymbol{I_{k}}$ denotes a $k \times k$ identity matrix. By directly calculation, we have
$$\overline{Q}^{\top}A(G, \varphi)Q = \left(
      \begin{array}{ccc}
        \boldsymbol{A} & \mathbf{0} & \mathbf{0} \\
        \mathbf{0} & -\overline{\mathbf{\beta}}^{\top}\xi & \mathbf{\gamma} \\
        \mathbf{0} &  \overline{\mathbf{\gamma}}^{\top} & \boldsymbol{B} \\
      \end{array}
    \right).
$$
Since $r(H, \varphi)=r((H, \varphi)+u)$, we have $-\overline{\mathbf{\beta}}^{\top}\xi=0$. Thus we have
$r(G, \varphi)=r(H, \varphi)+r((G, \varphi)-(H, \varphi))$.
\end{proof}

\begin{lemma} \label{L002}
Let $(C_{l}, \varphi)$ be a pendant complex unit gain cycle of a complex unit gain graph $(G, \varphi)$
with  $u$ be the only vertex of $(C_{l}, \varphi)$ of degree 3. Let $(H, \varphi)=(G, \varphi)-(C_{l}, \varphi)$ and $(G', \varphi)=(H, \varphi)+u$. If $Re((-1)^{\frac{l-1}{2}}\varphi(C_{l}, \varphi))=0$ and $l$ is odd, then
$$r(G, \varphi)=r(G', \varphi)+l-1.$$
\end{lemma}
\begin{proof} Note that $u$ is a cut vertex of $(G, \varphi)$ and $(P_{l-1}, \varphi)$ is a complex unit gain path as a component of $(G, \varphi)-u$.
By the fact that $Re((-1)^{\frac{l-1}{2}}\varphi(C_{l}, \varphi))=0$ and $l$ is odd, then by Lemmas \ref{LPN} and \ref{L12} one has that
$$r(P_{l-1}, \varphi)=r(C_{l}, \varphi)=l-1.$$

Then, by Lemma \ref{L001}, we have
$$r(G, \varphi)=r(G', \varphi)+r(P_{l-1}, \varphi)=r(G', \varphi)+l-1.$$
\end{proof}

By Lemma \ref{L12}, the following Lemma \ref{L50} can be obtained directly.

\begin{lemma} \label{L50}
The complex unit gain cycle $(C_{q}, \varphi)$ is lower-optimal if and only if either $\varphi(C_{q}, \varphi)=(-1)^{\frac{q}{2}}$ and $q$ is even or $Re((-1)^{\frac{q-1}{2}}\varphi(C_{q}, \varphi))=0$ and $q$ is odd.
\end{lemma}

\begin{lemma} \label{L000} Let $(G, \varphi)$ be a complex unit gain graph and $u$ be a vertex of $(G, \varphi)$ lying on a complex unit gain cycle. If $r(G, \varphi)= 2n-2c(G)-2\alpha(G)$, then each of the following holds.
\\
{\em(i)} $r(G, \varphi)=r((G, \varphi)-u)$;
\\
{\em(ii)} $(G, \varphi)-u$ is lower-optimal;
\\
{\em(iii)} $c(G)=c(G-u)+1$;
\\
{\em(iv)} $\alpha(G)=\alpha(G-u)$;
\\
{\em(v)} $u$ lies on just one complex unit gain cycle of $(G, \varphi)$ and $u$ is not a quasi-pendant vertex of $(G, \varphi)$.
\end{lemma}
\begin{proof}
In the proof arguments of Theorem \ref{T30} that justifies $r(G, \varphi)+2\alpha(G) \geq 2n-2c(G)$.
If both ends of (\ref{E0}) in the proof of Theorem \ref{T30} are the same, then all inequalities in (\ref{E0}) must be equalities, and so Lemma \ref{L000} (i)-(iv) are observed.

To prove (v). By Lemma \ref{L000} (iii) and Lemma \ref{L23}, we conclude that $u$ lies on just one
complex unit gain cycle of $(G, \varphi)$.
Suppose to the contrary that $u$ is a quasi-pendant vertex which adjacent to a pendant vertex $v$.
Then by Lemma \ref{L13}, we have
$$r((G, \varphi)-u)=r((G, \varphi)- \{ u, v  \})=r(G, \varphi)-2,$$
which is a contradiction to (i). This completes the proof of the lemma.
\end{proof}

\begin{lemma} \label{L51}
Let $(G, \varphi)$ be a complex unit gain graph and $(G_{1}, \varphi), (G_{2}, \varphi), \cdots, (G_{k}, \varphi)$ be all connected components of $(G, \varphi)$. Then $(G, \varphi)$ is lower-optimal if and only if $(G_{j}, \varphi)$ is lower-optimal for
each $j \in \{1, 2, \cdots, k \}$.
\end{lemma}
\begin{proof} (Sufficiency.) For each $i \in \{ 1, 2, \cdots, k \}$, one has that
$$r(G_{i}, \varphi)+2\alpha(G_{i})= 2|V(G_{i})|-2c(G_{i}).$$
Then, one has that
\begin{eqnarray*}
r(G, \varphi)&=&\sum\limits_{j=1}^{k}r(G_{j}, \varphi)\\
&=&\sum\limits_{j=1}^{k}[2|V(G_{i})|-2c(G_{i})-2\alpha(G_{i})]\\
&=&2|V(G)|-2c(G)-2\alpha(G).
\end{eqnarray*}

(Necessity.) Suppose to the contrary that there is a connected component of $(G, \varphi)$, say $(G_{1}, \varphi)$, which is not lower-optimal. By Theorem \ref{T30}, one has that
$$r(G_{1}, \varphi)+2\alpha(G_{1}) > 2|V(G_{1})|-2c(G_{1})$$
and for each $ j \in \{ 2, 3, \cdots, k \}$, we have
$$r(G_{j}, \varphi) +2\alpha(G_{j}) \geq 2|V(G_{j})|-2c(G_{j}).$$
Thus, one has that
$$r(G, \varphi)+2\alpha(G) > 2|V(G)|-2c(G),$$
a contradiction.
\end{proof}

\begin{lemma} \label{L56}
Let $u$ be a pendant vertex of a complex unit gain graph $(G, \varphi)$ and $v$ be the vertex which adjacent to $u$. Let $(G_{0}, \varphi)=(G, \varphi)-\{ u, v \}$. Then $(G, \varphi)$ is lower-optimal if and only if $v$ is not on any complex unit gain cycle of $(G, \varphi)$ and $(G_{0}, \varphi)$ is lower-optimal.

\end{lemma}
\begin{proof}
(Sufficiency.) Since $v$ is not on any complex unit gain cycle, by Lemma \ref{L23}, we have
$c(G)=c(G_{0})$.
By Lemmas \ref{L13} and \ref{L053}, one has that
$$r(G, \varphi)=r(G_{0}, \varphi)+2, \alpha(G)=\alpha(G_{0})+1.$$
Thus, one can get $(G, \varphi)$ is lower-optimal by the condition that $(G_{0}, \varphi)$ is lower-optimal.

(Necessity.) By Lemmas \ref{L13} and \ref{L000} and the condition that $(G, \varphi)$ is lower-optimal,
it can be checked that
$$r(G_{0}, \varphi)+2\alpha(G_{0})=2|V(G_{0})|-2c(G).$$
It follows from Theorem \ref{T30} that one has
$$r(G_{0}, \varphi)+2\alpha(G_{0}) \geq 2|V(G_{0})|-2c(G_{0}).$$
By the fact that $c(G_{0}) \leq c(G)$, then we have
$$c(G)=c(G_{0}), r(G_{0}, \varphi)+2\alpha(G_{0})=2|V(G_{0})|-2c(G_{0}).$$
Thus $(G_{0}, \varphi)$ is also lower-optimal and $v$ is not on any complex unit gain cycle of $(G, \varphi)$.
\end{proof}

\begin{lemma} \label{L55}
Let $(G, \varphi)$ be a complex unit gain graph obtained by joining a vertex $x$ of a complex unit gain cycle $(C_{l}, \varphi)$ by
an edge to a vertex $y$ of a complex unit gain connected graph $(K, \varphi)$. If $(G, \varphi)$ is lower-optimal, then the following properties hold for $(G, \varphi)$.

{\em(i)}  For each complex unit gain cycle $(C_{q}, \varphi)$ of $(G, \varphi)$, either $\varphi(C_{q}, \varphi)=(-1)^{\frac{q}{2}}$ and $q$ is even or $Re((-1)^{\frac{q-1}{2}}\varphi(C_{q}, \varphi))=0$ and $q$ is odd;

{\em(ii)} If $\varphi(C_{l}, \varphi)=(-1)^{\frac{l}{2}}$ and $l$ is even, then $r(G, \varphi)=l-2+r(K, \varphi)$ and $\alpha(G)=\frac{l}{2}+\alpha(K)$; if $Re((-1)^{\frac{l-1}{2}}\varphi(C_{l}, \varphi))=0$ and $l$ is odd, then
$r(G, \varphi)=l-1+r(K, \varphi)$ and $\alpha(G)=\frac{l-1}{2}+\alpha(K)$.

{\em(iii)} $(K, \varphi)$ is lower-optimal;

{\em(iv)} Let $(G', \varphi)$ be the induced complex unit gain subgraph of $(G, \varphi)$ with vertex set $V(K)\cup \{ x\}$. Then $(G', \varphi)$ is also lower-optimal;

{\em(v)} $\alpha(G')=\alpha(K)+1$ and $r(G', \varphi)=r(K, \varphi)$.
\end{lemma}
\begin{proof}
{\bf  (i):} We show (i) by induction on the order $n$ of $(G, \varphi)$. By Lemma \ref{L000}, $x$ can not
be a quasi-pendant vertex of $(G, \varphi)$, then $y$ is not an isolated vertex of $(G, \varphi)$. Then, $(K, \varphi)$ contains at least two vertices, i.e., $n\geq l+2$. If $n=l+2$, then $(K, \varphi)$ contains exactly
two vertices, without loss of generality, assume them be $y$ and $z$. Thus, one has that $(C_{l}, \varphi)=(G, \varphi)-\{ y, z \}$. By Lemma \ref{L56}, we have $(C_{l}, \varphi)$ is lower-optimal.
Then (i) follows from Lemma \ref{L50} directly.

Next, we consider the case of $n \geq l+3$. Suppose that (i) holds for every lower-optimal complex unit gain graph with order smaller than $n$. If $(K, \varphi)$ is a forest. Then $(G, \varphi)$ contains at least one pendant vertex. Let $u$ be
a pendant vertex of $(G, \varphi)$ and $v$ be the vertex which adjacent to $u$. By Lemma \ref{L000}, $v$ is not on $(C_{l}, \varphi)$. By Lemma \ref{L56}, one has that $(G, \varphi)-\{ u, v \}$ is lower-optimal.
By induction hypothesis to $(G, \varphi)-\{ u, v \}$, we have either $\varphi(C_{l}, \varphi)=(-1)^{\frac{l}{2}}$ and $l$ is even or $Re((-1)^{\frac{l-1}{2}}\varphi(C_{l}, \varphi))=0$ and $l$ is odd. Then (i) follows in this case.

If $(K, \varphi)$ contains cycles. Let $g$ be a vertex lying on a cycle of $(K, \varphi)$. By Lemma \ref{L000}, $(G, \varphi)-g$ is lower-optimal. Then, the induction hypothesis to $(G, \varphi)-g$ implies that either $\varphi(C_{l}, \varphi)=(-1)^{\frac{l}{2}}$ and $l$ is even or $Re((-1)^{\frac{l-1}{2}}\varphi(C_{l}, \varphi))=0$ and $l$ is odd. Let $s$ be a vertex lying on $(C_{l}, \varphi)$. By Lemma \ref{L000}, $(G, \varphi)-s$ is lower-optimal. Then, the induction hypothesis to $(G, \varphi)-s$ implies that for each cycle $(C_{q}, \varphi)$ of $(K, \varphi)$ either $\varphi(C_{q}, \varphi)=(-1)^{\frac{q}{2}}$ and $q$ is even or $Re((-1)^{\frac{q-1}{2}}\varphi(C_{q}, \varphi))=0$ and $q$ is odd. This completes the proof of (i).

Next we show (ii)-(v) according to the following two possible cases.

{\bf Case 1.} $\varphi(C_{l}, \varphi)=(-1)^{\frac{l}{2}}$ and $l$ is even.

{\bf  (ii):}
Since $x$ lies on a cycle of $(G, \varphi)$, by Lemmas \ref{L000}, \ref{LPN} and \ref{L053}, one has that
\begin{equation} \label{E6}
r(G, \varphi)=r((G, \varphi)-x)= r(P_{l-1}, \varphi)+r(K, \varphi)=l-2+r(K, \varphi)
\end{equation}
and
\begin{equation} \label{E7}
\alpha(G)=\alpha(G-x)= \alpha(P_{l-1})+\alpha(K)=\frac{l}{2}+\alpha(K).
\end{equation}

{\bf  (iii):} As $(C_{l}, \varphi)$ is a pendant cycle of $(G, \varphi)$, one has that
\begin{equation} \label{E8}
c(K)=c(G)-1.
\end{equation}
By (\ref{E6})-(\ref{E8}), we have
\begin{equation} \label{E9}
r(K, \varphi)+2\alpha(K)=2(n-l)-2c(K).
\end{equation}

{\bf  (iv):} Let $s$ be a vertex of $(C_{l}, \varphi)$ which adjacent to $x$. Then, by Lemmas \ref{L000}, \ref{L13} and \ref{L053}, we have
\begin{equation} \label{E10}
r(G, \varphi)=r((G, \varphi)-s)= l-2+r(G', \varphi)
\end{equation}
and
\begin{equation} \label{E11}
\alpha(G)=\alpha(G-s)= \frac{l-2}{2}+\alpha(G').
\end{equation}
It is obvious that $c(G)=c(G')+1$. Then from (\ref{E10})-(\ref{E11}), we have
\begin{eqnarray*}
r(G', \varphi)+2\alpha(G')&=&r(G, \varphi)+2\alpha(G)-2(l-2)\\
&=&2n-2c(G)-2(l-2)\\
&=&2(n-l+1)-2c(G').
\end{eqnarray*}

{\bf  (v):} Combining (\ref{E6}) and (\ref{E10}), one has that
$$r(K, \varphi)=r(G', \varphi).$$
From (\ref{E7}) and (\ref{E11}), we have
$$\alpha(K)+1=\alpha(G').$$

{\bf Case 2.} $Re((-1)^{\frac{l-1}{2}}\varphi(C_{l}, \varphi))=0$ and $l$ is odd.

{\bf  (ii):}
Since $x$ lies on a cycle of $(G, \varphi)$, by Lemmas \ref{L000}, \ref{L13} and \ref{L053}, one has that
\begin{equation} \label{E06}
r(G, \varphi)=r((G, \varphi)-x)= r(P_{l-1}, \varphi)+r(K, \varphi)=l-1+r(K, \varphi)
\end{equation}
and
\begin{equation} \label{E07}
\alpha(G)=\alpha(G-x)= \alpha(P_{l-1})+\alpha(K)=\frac{l-1}{2}+\alpha(K).
\end{equation}

{\bf  (iii):} As $C_{l}$ is a pendant cycle of $(G, \varphi)$, one has that
\begin{equation} \label{E08}
c(K)=c(G)-1.
\end{equation}
By (\ref{E06})-(\ref{E08}), we have
\begin{equation} \label{E09}
r(K, \varphi)+2\alpha(K)=2(n-l)-2c(K).
\end{equation}

{\bf  (iv)} and {\bf  (v):} By Lemma \ref{L002}, we have
\begin{equation} \label{E010}
r(G, \varphi)= l-1+r(G', \varphi).
\end{equation}
Then, by (\ref{E06}) and (\ref{E010}) we have
\begin{equation} \label{E011}
r(G', \varphi)=r(K, \varphi).
\end{equation}
By (\ref{E09}) and Theorem \ref{T30}, one has that
\begin{eqnarray*}
2\alpha(K)&=&2(n-l)-r(K, \varphi)-2c(K)\\
&=&2(n-l+1)-r(G', \varphi)-2c(K)-2\\
&=&2(n-l+1)-r(G', \varphi)-2c(G')-2\\
&\leq&2\alpha(G')-2.
\end{eqnarray*}
Thus, we have $\alpha(K)\leq \alpha(G')-1$.
On the other hand, by Lemma \ref{L054}, we have $\alpha(K) \geq \alpha(G')-1$.
Hence,
\begin{equation} \label{E012}
\alpha(K) = \alpha(G')-1.
\end{equation}

It is obvious that $c(G')=c(K)$. Combing (\ref{E09}), (\ref{E011}) and (\ref{E012}), one has that
$$r(G', \varphi)+2\alpha(G')=2(n-l+1)-2c(G').$$
This implies (iv). Moreover, equalities (\ref{E011}) and (\ref{E012}) implies (v).

This completes the proof.

\end{proof}

\begin{lemma} \label{L58}
Let $(G, \varphi)$ be a lower-optimal complex unit gain graph. Then
$\alpha(G)=\alpha(T_{G})+\sum_{C \in \mathscr{C}_{G}}\lfloor\frac{|V(C)|}{2}\rfloor-c(G)$.
\end{lemma}

\begin{proof}

 We argue by induction on the order $n$ of $G$ to show the lemma. If $n=1$, then the lemma
holds trivially. Next, we consider the case of $n \geq 2$. Suppose that the result holds for every lower-optimal complex unit gain graph with order smaller than $n$.

If $E(T_{G})=0$, i.e., $T_{G}$ is an empty graph, then each component of $(G, \varphi)$ is a cycle or an isolated vertex.  For each cycle $C_{l}$, it is routine to check that
$\alpha(C_{l})= \lfloor \frac{l}{2} \rfloor$.  Then the lemma follows.

If $E(T_{G}) \geq 1$. Then $T_{G}$ contains at least one pendant vertex, say $x$. If $x$ is also a pendant vertex in $(G, \varphi)$, then $(G, \varphi)$ contains a pendant vertex. If $x$ is a vertex obtained by contracting a cycle of $(G, \varphi)$, then $(G, \varphi)$ contains a pendant cycle.
Then we will deal with the following two cases.

{\bf Case 1.}  $x$ is also a pendant vertex in $(G, \varphi)$.

Let $y$ be the unique neighbour of $x$ and
$(G_{0}, \varphi)= (G, \varphi)-\{ x, y \}$. By Lemma \ref{L56}, one has that $y$ is not on any cycle of $(G, \varphi)$
and $(G_{0}, \varphi)$ is lower-optimal. Furthermore, it is obvious that $c(G)=c(G_{0})$. By induction hypothesis, we have

(a) $\alpha(G_{0})=\alpha(T_{G_{0}})+\sum_{C \in \mathscr{C}_{G_{0}}}\lfloor\frac{|V(C)|}{2}\rfloor-c(G_{0})$.

Sine $x$ is a pendant vertex of $(G, \varphi)$ and $y$ is a quasi-pendant vertex which is not in any cycle of $(G, \varphi)$, $x$ is a pendant vertex of $T_{G}$ and $y$ is a quasi-pendant vertex of $T_{G}$. Moreover,
$T_{G_{0}}=T_{G}-\{ x, y \}$. Thus, by Lemma \ref{L053} and assertion (a), we have
\begin{eqnarray*}
\alpha(G)&=&\alpha(G_{0})+1\\
&=&\alpha(T_{G_{0}})+\sum_{C \in \mathscr{C}_{G_{0}}}\lfloor\frac{|V(C)|}{2}\rfloor-c(G_{0})+1\\
&=&\alpha(T_{G})-1+\sum_{C \in \mathscr{C}_{G_{0}}}\lfloor\frac{|V(C)|}{2}\rfloor-c(G_{0})+1\\
&=&\alpha(T_{G})+\sum_{C \in \mathscr{C}_{G}}\lfloor\frac{|V(C)|}{2}\rfloor-c(G).
\end{eqnarray*}
Thus, the result holds in this case.

{\bf Case 2.} $x$ lies on a pendant cycle.

Let $x$ lies on a pendant cycle $C_{q}$.
In this case, one can suppose that $x$ is the unique vertex of $C_{q}$ of degree 3. Let $K=G-C_{q}$ and $(G_{1}, \varphi)$ be the induced complex unit gain subgraph of $(G, \varphi)$ with vertex set $V(K)\cup \{ x\}$. By Lemma \ref{L55} (iv), one has that $(G_{1}, \varphi)$ is lower-optimal. By induction hypothesis, we have

(c)  $\alpha(G_{1})=\alpha(T_{G_{1}})+\sum_{C \in \mathscr{C}_{G_{1}}}\lfloor\frac{|V(C)|}{2}\rfloor-c(G_{1})$.

It can be checked that
$$\mathscr{C}_{G}=\mathscr{C}_{G_{1}} \cup C_{q}=\mathscr{C}_{K} \cup C_{q}.$$
Moreover, one has that
\begin{equation} \label{E13}
\sum_{C \in \mathscr{C}_{G}}\lfloor\frac{|V(C)|}{2}\rfloor=\sum_{C \in \mathscr{C}_{G_{1}}}\lfloor\frac{|V(C)|}{2}\rfloor+\lfloor\frac{q}{2}\rfloor=\sum_{C \in \mathscr{C}_{K}}\lfloor\frac{|V(C)|}{2}\rfloor+\lfloor\frac{q}{2}\rfloor.
\end{equation}

Since $C_{q}$ is a pendant cycle of $(G, \varphi)$, it is obvious that
\begin{equation} \label{E14}
c(G_{1})=c(K)=c(G)-1.
\end{equation}
By Lemma \ref{L55} (v), one has that
\begin{equation} \label{E15}
\alpha(G_{1})=\alpha(K)+1.
\end{equation}
Note that
\begin{equation} \label{E150}
T_{G_{1}}=T_{G}.
\end{equation}
By Lemma \ref{L55} (ii) and (\ref{E13})-(\ref{E150}), one has that
\begin{eqnarray*}
\alpha(G)&=&\alpha(K)+\lfloor\frac{p}{2}\rfloor\\
&=&\alpha(G_{1})+\lfloor\frac{p}{2}\rfloor-1\\
&=&\alpha(T_{G_{1}})+\sum_{C \in \mathscr{C}_{G_{1}}}\lfloor\frac{|V(C)|}{2}\rfloor-c(G_{1})+\lfloor\frac{p}{2}\rfloor-1\\
&=&\alpha(T_{G})+\sum_{C \in \mathscr{C}_{G}}\lfloor\frac{|V(C)|}{2}\rfloor-c(G_{1})-1\\
&=&\alpha(T_{G})+\sum_{C \in \mathscr{C}_{G}}\lfloor\frac{|V(C)|}{2}\rfloor-c(G).
\end{eqnarray*}

This completes the proof.
\end{proof}

\noindent
{\bf The proof of Theorem \ref{T50}.}
(Sufficiency.) We proceed by induction on the order $n$ of $(G, \varphi)$. If $n=1$, then the result holds trivially. Therefore we assume that $(G, \varphi)$ is a complex unit gain graph with order $n \geq 2$ and satisfies (i)-(iii). Suppose that any complex unit gain graph of order smaller than $n$ which satisfes (i)-(iii) is
lower-optimal. Since the cycles (if any) of $(G, \varphi)$ are pairwise vertex-disjoint, $(G, \varphi)$ has
exactly $c(G)$ cycles, i.e., $|\mathscr{O}_{G}|=c(G)$.

If $E(T_{G})=0$, i.e., $T_{G}$ is an empty graph, then each component of $(G, \varphi)$ is a cycle or an isolated vertex. By (ii) and Lemma \ref{L50}, we have $(G, \varphi)$ is lower-optimal.

If $E(T_{G}) \geq 1$. Then $T_{G}$ contains at least one pendant vertex. By (iii), one has that
$$\alpha(T_{G})=\alpha([T_{G}])+c(G)=\alpha(T_{G}-\mathscr{O}_{G})+c(G)=\alpha(T_{G}-\mathscr{O}_{G})+|\mathscr{O}_{G}|.$$
Thus, by Lemma \ref{L055} (ii), there exists a pendent vertex of $T_{G}$ which is not in $\mathscr{O}_{G}$. Then, $(G, \varphi)$ contains at least one pendant vertex, say $u$. Let $v$ be the unique neighbour of $u$ and let
$(G_{0}, \varphi)=(G, \varphi)-\{ u, v \}$. It is obvious that $u$ is a pendant vertex of $T_{G}$ adjacent to $v$ and
$T_{G_{0}}=T_{G}-\{ u, v \}$. By Lemma \ref{L053}, one has that
$$\alpha(T_{G})=\alpha(T_{G}-v)=\alpha(T_{G}-\{ u, v \})+1.$$

{\bf Claim.} $v$ does not lie on any cycle of $(G, \varphi)$.

By contradiction, assume that $v$ lies on a cycle of $(G, \varphi)$. Then $v$ is in $\mathscr{O}_{G}$. Note that the size of $\mathscr{O}_{G}$ is $c(G)$. Then, $H:=(T_{G}-v) \cup K_{1}$ is a spanning subgraph of $T_{G}$.
Delete all the edges $e$ in $H$ such that $e$ contains at least one end-vertex in $\mathscr{O}_{G} \backslash \{ v \}$.
Thus, the resulting graph is $[T_{G}] \cup c(G)K_{1}$. By Lemma \ref{L054}, one has that
$$\alpha([T_{G}] \cup c(G)K_{1}) \geq \alpha((T_{G}-v) \cup K_{1}),$$
that is,
$$\alpha([T_{G}])+c(G) \geq \alpha(T_{G}-v)+1.$$
Then, we have
$$\alpha([T_{G}]) \geq \alpha(T_{G}-v)+1-c(G)=\alpha(T_{G})+1-c(G),$$
a contradiction to (iii). This completes the proof of the claim.

Thus, $v$ does not lie on any cycle of $(G, \varphi)$. Moreover, $u$ is also a pendant vertex of $[T_{G}]$ which adjacent to $v$ and $[T_{G_{0}}]=[T_{G}]-\{ u, v \}$. By Lemma \ref{L053}, one has that
$$\alpha([T_{G}])=\alpha([T_{G_{0}}])+1.$$
It is routine to checked that $c(G)=c(G_{0}).$
Thus,
\begin{eqnarray*}
\alpha(T_{G_{0}})&=&\alpha(T_{G})-1\\
&=&\alpha([T_{G}])+c(G)-1\\
&=&\alpha([T_{G_{0}}])+1+c(G)-1\\
&=&\alpha([T_{G_{0}}])+c(G_{0}).
\end{eqnarray*}

Combining the fact that all cycles of $(G, \varphi)$ belong to $(G_{0}, \varphi)$, one has that
$(G_{0}, \varphi)$ satisfies all the conditions (i)-(iii). By induction hypothesis, we have $(G_{0}, \varphi)$
is lower-optimal. By Lemma \ref{L56}, we have $(G, \varphi)$ is lower-optimal.

(Necessity.) Let $(G, \varphi)$ be a lower-optimal complex unit gain graph. If $(G, \varphi)$ is a complex unit gain acyclic graph, then
(i)-(iii) holds directly. So one can suppose that $(G, \varphi)$ contains cycles. By Lemma \ref{L000} (v) and \ref{L55} (i), one has that the cycles (if any) of $(G, \varphi)$ are pairwise vertex-disjoint and for each cycle $(C_{l}, \varphi)$ of $(G, \varphi)$, either $\varphi(C_{l}, \varphi)=(-1)^{\frac{l}{2}}$ and $l$ is even or $Re((-1)^{\frac{l-1}{2}}\varphi(C_{l}, \varphi))=0$ and $l$ is odd. This completes the proof of (i) and (ii).

Next, we argue by induction on the order $n$ of $(G, \varphi)$ to show (iii). Since $(G, \varphi)$ contains cycles, $n \geq 3$. If $n=3$, then $(G, \varphi)$ is a 3-cycle and (iii) holds trivially. Therefore we assume that $(G, \varphi)$ is a lower-optimal complex unit gain graph with order $n \geq 4$.
Suppose that (iii) holds for all lower-optimal complex unit gain graphs of order smaller than $n$.

If $E(T_{G})=0$, i.e., $T_{G}$ is an empty graph, then each component of $(G, \varphi)$ is a cycle or an isolated vertex. Then, (iii) follows.

If $E(T_{G}) \geq 1$. Then $T_{G}$ contains at least one pendant vertex, say $x$. If $x$ is also a pendant vertex in $(G, \varphi)$, then $(G, \varphi)$ contains a pendant vertex. If $x$ is a vertex obtained by contracting a cycle of $(G, \varphi)$, then $(G, \varphi)$ contains a pendant cycle.
Then we will deal with (iii) with the following two cases.

{\bf Case 1.} $x$ is a pendant vertex of $(G, \varphi)$.

Let $y$ be the unique neighbour of $x$ and
$(G_{1}, \varphi)= (G, \varphi)-\{ x, y \}$. By Lemma \ref{L56}, one has that $y$ is not on any cycle of $(G, \varphi)$
and $(G_{1}, \varphi)$ is lower-optimal. By induction hypothesis, we have
$$\alpha(T_{G_{1}})=\alpha([T_{G_{1}}])+c(G_{1}).$$

Note that $x$ is also a pendant vertex of $T_{G}$ which adjacent to $y$, then $T_{G_{1}}=T_{G}-\{ x, y \}$,
$[T_{G_{1}}]=[T_{G}]-\{ x, y \}$ and $c(G)=c(G_{1})$.
By Lemma \ref{L053}, it can be checked that
$$\alpha(T_{G})=\alpha([T_{G}])+c(G).$$

The result follows.

{\bf Case 2.} $(G, \varphi)$ contains a pendant cycle.

Let $(C_{q}, \varphi)$ be a pendant complex unit gain cycle of $(G, \varphi)$ and $(K, \varphi)=(G, \varphi)-(C_{q}, \varphi)$. By Lemma \ref{L55} (ii), one has that
$(K, \varphi)$ is lower-optimal. By induction hypothesis, we have
\begin{equation} \label{E20}
\alpha(T_{K})=\alpha([T_{K}])+c(K).
\end{equation}
In view of Lemma \ref{L55} (ii), one has
\begin{equation} \label{E21}
\alpha(G)=\alpha(K)+\lfloor\frac{q}{2}\rfloor.
\end{equation}
Since $\mathscr{C}_{G}=\mathscr{C}_{K} \cup C_{q}$.
Then, we have
\begin{equation} \label{E22}
\sum_{C \in \mathscr{C}_{G}}\lfloor\frac{|V(C)|}{2}\rfloor=\sum_{C \in \mathscr{C}_{K}}\lfloor\frac{|V(C)|}{2}\rfloor+\lfloor\frac{q}{2}\rfloor.
\end{equation}
Since  $(G, \varphi)$ and  $(K, \varphi)$ are lower-optimal, by Lemma \ref{L58}, we have
\begin{equation} \label{E23}
\alpha(T_{G})=\alpha(G)-\sum_{C \in \mathscr{C}_{G}}\lfloor\frac{|V(C)|}{2}\rfloor+c(G)
\end{equation}
and
\begin{equation} \label{E24}
\alpha(T_{K})=\alpha(K)-\sum_{C \in \mathscr{C}_{K}}\lfloor\frac{|V(C)|}{2}\rfloor+c(K).
\end{equation}
It is routine to check that $c(G)=c(K)+1$. Then combining (\ref{E20})-(\ref{E24}), we have
\begin{eqnarray*}
\alpha(T_{G})&=&\alpha(G)-\sum_{C \in \mathscr{C}_{G}}\lfloor\frac{|V(C)|}{2}\rfloor+c(G)\\
&=&\alpha(K)+\lfloor\frac{q}{2}\rfloor-\sum_{C \in \mathscr{C}_{G}}\lfloor\frac{|V(C)|}{2}\rfloor+c(G)\\
&=&\alpha(K)-\sum_{C \in \mathscr{C}_{K}}\lfloor\frac{|V(C)|}{2}\rfloor+c(G)\\
&=&\alpha(K)-\sum_{C \in \mathscr{C}_{K}}\lfloor\frac{|V(C)|}{2}\rfloor+c(K)+1\\
&=&\alpha(T_{K})+1.
\end{eqnarray*}
Note that
\begin{equation} \label{E26}
[T_{G}]\cong [T_{K}].
\end{equation}
Then, in view of (\ref{E20}) and (\ref{E26}), one has that
\begin{eqnarray*}
\alpha(T_{G})&=&\alpha(T_{K})+1\\
&=&\alpha([T_{K}])+c(K)+1\\
&=&\alpha([T_{G}])+c(G).\\
\end{eqnarray*}

This completes the proof.
$\square$

\section*{Acknowledgments}

This work was supported by the National Natural Science Foundation of China
(No. 11731002), the Fundamental Research Funds for the Central Universities (No. 2016JBZ012) and the 111 Project of China (B16002).

\end{document}